\documentclass[12pt]{amsart}
\usepackage{xcolor}
\usepackage{amsmath,amssymb,setspace,nicefrac, yhmath, amscd,eucal,dsfont}
\usepackage[active]{srcltx}
\usepackage[colorlinks, linkcolor=red, citecolor=blue, urlcolor=blue, hypertexnames=true]{hyperref}
\usepackage{enumitem}
\usepackage{amsrefs}
\usepackage{tikz-cd}
\usepackage{soul}
\allowdisplaybreaks
\usepackage[all]{xy}
\usepackage{hyperref}

\newcommand{\N}{{\mathbb N}}

\newcommand{\cU}{{\mathcal U}}

\newtheorem{thm}{Theorem}[section]
\newtheorem{cor}[thm]{Corollary}
\newtheorem{lemma}[thm]{Lemma}
\newtheorem{question}[thm]{Question}
\newtheorem{definition}[thm]{Definition}
\newtheorem{example}[thm]{Example}
\newtheorem{remark}[thm]{Remark}

\theoremstyle{definition}

\title[Numerical Index, Convergence and Ultraproducts]{On the Convergence of Numerical Index via Operator Openings and Ultraproducts}
\author[Monika]{Monika}\address{Hampton University, Hampton, VA, USA}\email{fnu.monika@hamptonu.edu}
\author[Amrutam]{Tattwamasi Amrutam}\address{Institute of Mathematics of the Polish Academy of Sciences, Ul.S'niadeckich 8, 00-656 Warszawa, Poland}\email{tattwamasiamrutam@gmail.com} \author[Dey]{Priyadarshi Dey}\address{Millsaps College, Jackson, Mississippi, USA}\email{priyadarshid4@gmail.com, deyp@millsaps.edu}

\date{\today}

\begin{document}
\begin{abstract}
The numerical index of a Banach space is a geometric constant relating the numerical radius of bounded linear operators to their standard operator norm. In this paper, we study the continuity of the numerical index under two distinct notions of subspace convergence. First, we establish a full limit theorem in the operator opening topology: if $\{X_n\}_{n \in \mathbb{N}}$ and $X$ are closed subspaces of a Banach space $Y$ with $X_n \to X$ in the operator opening, then $\lim_{n \to \infty} n(X_n) = n(X)$. Second, we develop ultraproduct methods for the numerical index, proving that the numerical radius is exactly preserved by ultraproduct operators, i.e., $v((T_n)_{\mathcal{U}}) = \lim_{\mathcal{U}} v(T_n)$. As a consequence, we show that $n(X_{\mathcal{U}}) \le n(X)$ for every ultrapower $\mathcal{U}$.
\end{abstract}
\maketitle
\section{Introduction}
\newtheorem{thmx}{\textbf{Theorem}}
\renewcommand{\thethmx}{\Alph{thmx}}
The numerical index of a Banach space $X$, denoted $n(X)$, is a geometric constant that relates the numerical radius of bounded linear operators on $X$ to their standard operator norm. First introduced by Lumer in 1968, the numerical index has since played a central role at the intersection of Banach space geometry and operator theory. However, explicit computation of $n(X)$ remains difficult, particularly for infinite-dimensional spaces. 
 
A natural approach for addressing this difficulty is approximation: given a Banach space $X$, can we approximate it by a family of simpler subspaces $\{X_m\}$ and compute $n(X)$ via the limit of the numerical indices $n(X_m)$? This approach has been explored in the literature, most notably by Mart\'in et al.~\cite{martin2011numerical}, who obtained limsup inequalities under the existence of norm-one complemented subspaces. Subsequently, Aksoy and Lewicki~\cite{aksoy2013limit} established full limit theorems under Local and Global Characterization Conditions (LCC and GCC). Conditions such as LCC and GCC depend on the existence of norm-one projections compatible with norming functionals. While these hypotheses are satisfied in many classical settings, they exclude a broad class of Banach spaces in which such projections are difficult to construct, such as hereditarily indecomposable spaces \cite{HIspaces,LindenstraussTzafriri1977}.
 
In this paper, we propose two approaches that overcome these limitations, neither of which requires bounded projections.
 
Our first approach replaces the projection-based framework with a topological one. Rather than projecting from $X$ onto $X_m$, we ask how close $X_m$ is as a subspace of $X$, measured by the \textit{operator opening}. For closed subspaces $Y$, $Z$ of a Banach space $X$, the operator opening is defined by 
\[
r(Y,Z) = \max\{r_0(Y,Z),\, r_0(Z,Y)\},
\]
where
\[
r_0(Y,Z) = \inf\bigl\{\|C - I\| : C \in GL(X),\, C(Y) = Z\bigr\},
\]
with the convention $r_0(Y,Z)=1$ when no such invertible operator exists. The topology induced by the operator opening majorizes the topology induced by the geometric (gap) opening~\cite[Theorem~4.2(c)]{ostrovskii1994topologies}. Our first main result establishes the full continuity of the numerical index under this topology.
\begin{thmx}
\label{thm:operatorgapconvergence}
  Let $Y$ be a Banach space and let $\{X_n\}_{n \in \mathbb{N}}$ and $X$ be closed subspaces of $Y$. If $X_n \to X$ in the operator opening topology, then 
    \[
        \lim_{n \to \infty} n(X_n) = n(X).
    \]  
\end{thmx}
The proof proceeds in two steps. First, we establish norm bounds on the invertible operators present in the operator opening condition (Lemma~\ref{lem:bounds}). Second, we apply the Bishop-Phelps-Bollob\'{a}s theorem to approximate the norming functionals. Together, these tools allow us to prove both the inequalities $\limsup_{n\to\infty} n(X_n) \le n(X)$ and $\liminf_{n\to\infty} n(X_n) \ge n(X)$, yielding the full limit. The core advantage of the operator-opening topology is the existence of near-identity isomorphisms, which allow us to conjugate optimal operators between $X_n$ and $X$, thereby circumventing the need for projections or operator extensions.
 
While the operator opening topology provides a clean, universal framework, it is naturally stronger than the classical gap (or geometric) opening topology. A fundamental open question is whether geometric convergence in the gap topology alone is sufficient to guarantee the continuity of the numerical index. We address this question via ultraproduct methods.
 
Our second main result establishes the exact preservation of the numerical radius under the ultraproduct construction.
\begin{thmx}
\label{thm:ultraproductintro}
Let $\{X_n\}_{n \in \mathbb{N}}$ be a sequence of Banach spaces and let $\cU$ be a free ultrafilter on $\mathbb{N}$. For any uniformly bounded sequence of operators $T_n \in \mathcal{L}(X_n)$,
\[
    v\big((T_n)_\cU\big) = \lim_\cU v(T_n).
\]
\end{thmx}
A key consequence of this result is the inequality $n(\prod_\cU X_n) \le \lim_\cU n(X_n)$, and in particular $n(X_\cU) \le n(X)$ for every ultrapower. When combined with the observation that gap topology convergence $X_n \to X$ inside an ambient space $Y$ implies the ultraproduct identification $\prod_\cU X_n = X_\cU$ inside $Y_\cU$ (Remark~\ref{rem:gap_identification}), we obtain the following reduction.
\begin{thmx}
\label{thm:reductionintro}
Let $Y$ be a Banach space and let $\{X_n\}_{n \in \mathbb{N}}$ and $X$ be closed subspaces of $Y$ with $X_n \to X$ in the gap topology. Then for every free ultrafilter $\cU$,
\[
    n(X_\cU) \le \lim_\cU n(X_n).
\]
Consequently, if $n(X_\cU) = n(X)$ holds for the space $X$, then $n(X) \le \liminf_{n \to \infty} n(X_n)$.
\end{thmx}
This reduces the gap topology convergence problem to the following purely intrinsic question about ultrapowers (see Question~\ref{q:ultrapower}). We note that the answer is affirmative when $X$ is lush (since lushness is preserved by ultraproducts~\cite{BKMM2009}, giving $n(X_\cU) = n(X) = 1$) and trivially when $n(X) = 0$. For general Banach spaces, the question remains open.
 
\subsection*{Acknowledgment} T.A. was partially supported by the Simons Foundation grant (award no. SFI-MPS-T-Institutes-00010825) and from State Treasury funds as part of a task commissioned by the Minister of Science and Higher Education under the project "Organization of the Simons Semesters at the Banach Center - New Energies in 2026-2028" (agreement no. MNiSW/2025/DAP/491).
\section{Preliminaries}
Given a Banach space $X$, we write $B_{X}, S_{X}$, and $X^*$ to denote the unit ball, unit sphere, and the dual space of $X$ and define
\[\Pi(X):= \{(x,x^*): x\in S_{X}, x^{*}\in S_{X^*}, x^*(x)=1\}.\]
We briefly recall the notions of numerical range, numerical radius, and numerical index and refer the readers to the classical monographs by Bonsall and Duncan \cite{BonsallDuncan1971,BonsallDuncan1973} for more details. 
\begin{definition}[Numerical range, Numerical radius, and Numerical index]
Let $X$ be a Banach space over $\mathbb{K}$ and let $T\in \mathcal{L}(X)$ be a bounded linear operator.  
The numerical range of $T$ is the set
\[
V(T)
:=\left\{x^*(Tx):\ x\in S_X,\ x^*\in S_{X^{\ast}}, x^{\ast}(x)=1\right\}.
\]Then, the numerical radius of $T$ is the supremum of absolute values of the elements in the numerical range, i.e.,
\[
v(T)
:=\sup\left\{| \lambda|:\ \lambda\in V(T)\right\}.
\]
Finally, the numerical index of $X$ is defined by taking the infimum of the numerical radius of norm one operators on $X$, that is
\[
n(X)
:=\inf\left\{v(T):\ T\in \mathcal{L}(X),\ \|T\|=1\right\}.
\]
\end{definition}
Finally, we need the following easy modification of the well-known Bishop-Phelps-Boll\'{o}bas Theorem~\cite[Theorem~1]{Bollob1970}. We record it for completeness.
\begin{lemma}\label{lem:bpb_sequence}
    Let $X$ be a Banach space. Let $\varepsilon>0$. Suppose there exists $u\in B_X$ and $u^*\in B_{X^*}$ such that 
    \[
       \text{Re }  u^*(u) = 1 - \delta,
    \]
    where $|\delta|<\frac{\varepsilon^2}{2}$. Then there exists a state pair $(y, y^*) \in \Pi(X)$ such that 
    \[
         \|u - y\| <\varepsilon \quad \text{and} \quad  \|u^* - y^*\|<\varepsilon.
    \]
\end{lemma}
 
\begin{proof}
Let $\varepsilon>0$. Since  $|\delta|<\frac{\varepsilon^2}{2}$  and $\text{Re } u^*(u) = 1 - \delta$, 
    \[
        \text{Re } u^*(u) = 1 - \delta \ge 1 - |\delta|>1-\frac{\varepsilon^2}{2}.
    \]
     By the Bishop-Phelps-Boll\'{o}bas Theorem for the closed unit ball \cite[Theorem 1.2]{dantas2022bishopold}, there exists a pair $(y, y^*) \in \Pi(X)$ with $\|u - y\| < \varepsilon$ and $\|u^* - y^*\| < \varepsilon$. The claim follows.
\end{proof}
 
\subsection{Gap Topology}
\label{subsec:gaptopology}
For a Banach space $X$, let $G(X)$ denote the collection of all closed subspaces of $X$. For any $x \in X$ and subset $A \subseteq X$, recall that the distance is defined as $\text{dist}(x, A) = \inf_{a \in A} \|x - a\|$.
 
For $Y, Z \in G(X)$, the geometric (or gap) opening between $Y$ and $Z$ is defined as
\[
    Q(Y,Z) = \max \left\{ \sup_{y \in S_Y} \text{dist}(y, S_Z), \sup_{z \in S_Z} \text{dist}(z, S_Y) \right\}.
\]
For a sequence $Y_m \in G(X)$ and $Y \in G(X)$, we say that $Y_m \to Y$ in the gap topology (in the sense of Ostrovskii \cite{ostrovskii1994topologies}) if $\lim_{m \to \infty} Q(Y_m, Y) = 0$. If a sequence of subspaces $X_m \to X$ in the gap topology inside an ambient space $Y$, the condition $Q(X_m, X) \to 0$ yields two crucial approximation properties for their unit spheres that we will use extensively.
    \begin{enumerate}
        \item Because $\sup_{x \in S_X} \text{dist}(x, S_{X_m}) \to 0$, any fixed vector on the unit sphere of the limit space can be approximated by vectors from the unit spheres of the sequence. Specifically, for any fixed $x_0 \in S_X$, there exists a sequence $x_m \in S_{X_m}$ such that $\|x_m - x_0\|_Y \to 0$ as $m \to \infty$.
        \item Because $\sup_{z \in S_{X_m}} \text{dist}(z, S_X) \to 0$, any sequence of vectors picked from the unit spheres of $X_m$ eventually gets arbitrarily close to the unit sphere of $X$. Specifically, for any sequence $x_m \in S_{X_m}$, there exists a corresponding sequence $u_m \in S_X$ such that $\|x_m - u_m\|_Y \to 0$ as $m \to \infty$.
    \end{enumerate}
 
\subsection{Operator Opening Topology}    
Let $X$ be a Banach space and let $G(X)$ denote the set of closed subspaces of $X$. By $GL(X)$ we denote the group of all invertible linear operators on $X$. For $Y, Z \in G(X)$, let
$$r_0(Y, Z) = \inf\{\|C - I\| : C \in GL(X), C(Y) = Z\},$$
if the set over which the infimum is taken is not empty, and $r_0(Y, Z) = 1$ otherwise.
The \textit{operator opening} between $Y$ and $Z$ is defined by
$$r(Y, Z) = \max\{r_0(Y, Z), r_0(Z, Y)\}.$$
 
The topology induced by the operator opening majorizes the topology induced by the geometric opening~\cite[Theorem~4.2(c)]{ostrovskii1994topologies}. We show below that if $X_m\to X$ in the operator opening topology, then $\lim_{m \to \infty} n(X_m) = n(X)$.
\begin{lemma}\label{lem:bounds} Let $Y$ be a Banach space. Let $X$ and $Z$ be closed subspaces of $Y$. Suppose that there exists $C\in GL(Y)$ with $C(X)=Z$ and $\|C-I\|<\frac{\eta}{2}$ for some $0<\eta<2$. Then, 
\[\frac{1}{\|C\|\|C^{-1}\|}\ge \frac{2-\eta}{2+\eta}.\]
\end{lemma}
\begin{proof}
To begin with, since $\|C-I\|<1$, the Neumann series gives $C^{-1}=\sum_{k=0}^{\infty}(I-C)^{k}$ and hence
\[\|C^{-1}\|\le \frac{1}{1-\|I-C\|}\le  \frac{2}{2-\eta}.\]
Moreover, by the triangle inequality, we get that
\[\|C\| \le \|I\|+\|C-I\| < 1+\frac{\eta}{2}=\frac{2+\eta}{2},\]
which in turn implies that
\[\|C\|\|C^{-1}\|\le \frac{2+\eta}{2}\cdot \frac{2}{2-\eta}=\frac{2+\eta}{2-\eta},\]
which is equivalent to the one claimed.
\end{proof}
\subsection{Ultraproducts of Banach Spaces}
We briefly recall the Banach space ultraproduct construction and refer the reader to Heinrich~\cite{heinrich1980} for a comprehensive treatment.
 
\begin{definition}[Banach space ultraproduct]
Let $\{X_n\}_{n \in \N}$ be a sequence of Banach spaces and let $\cU$ be a free ultrafilter on $\N$. The \textit{ultraproduct} of $\{X_n\}$ with respect to $\cU$ is the quotient $\prod_\cU X_n = \ell_\infty(\{X_n\}) / c_{0,\cU}(\{X_n\})$, where $c_{0,\cU}(\{X_n\}) = \{ (x_n) \in \ell_\infty(\{X_n\}) : \lim_\cU \|x_n\| = 0 \}$. The equivalence class of $(x_n)$ is denoted $(x_n)_\cU$, and the quotient norm satisfies $\|(x_n)_\cU\| = \lim_\cU \|x_n\|$. When all $X_n = X$, the ultraproduct is called the \textit{ultrapower} $X_\cU$, and the \textit{diagonal embedding} $\delta(x) = (x, x, \ldots)_\cU$ is an isometry.
\end{definition}
 
\begin{definition}[Ultraproduct of operators]
For $T_n \in \mathcal{L}(X_n)$ with $\sup_n \|T_n\| < \infty$, the \textit{ultraproduct operator} $(T_n)_\cU((x_n)_\cU) = (T_n x_n)_\cU$ satisfies $\|(T_n)_\cU\| = \lim_\cU \|T_n\|$.
\end{definition}
 
\begin{remark}[The dual of an ultraproduct]\label{rem:dual_ultraproduct}
The natural map $\Phi \colon \prod_\cU X_n^* \to (\prod_\cU X_n)^*$ defined by $\Phi((f_n)_\cU)((x_n)_\cU) = \lim_\cU f_n(x_n)$ is an isometric embedding. In general $\Phi$ is not surjective, but $\Phi(\prod_\cU X_n^*)$ is always a $1$-norming subspace: for every $(x_n)_\cU \in \prod_\cU X_n$,
\begin{equation}\label{eq:norming}
    \|(x_n)_\cU\| = \sup\big\{ |\lim_\cU f_n(x_n)| : (f_n)_\cU \in \prod_\cU X_n^*,\; \lim_\cU \|f_n\| \le 1 \big\}.
\end{equation}
Indeed, choosing $f_n \in S_{X_n^*}$ with $f_n(x_n) = \|x_n\|$ achieves the supremum.
\end{remark}
 
\begin{remark}\label{rem:gap_identification}
If $X_n \to X$ in gap topology inside an ambient Banach space $Y$, then $\prod_\cU X_n \cong X_\cU$ isometrically inside $Y_\cU$. For the inclusion $(\subseteq)$: let $(x_n)_{\cU} \in \prod_{\cU}X_n$ with $x_n \in X_n$. We may assume, without loss of generality that $x_n \ne 0$ for $\cU$-almost all $n$. Let $x_n=\|x_n\| \cdot \frac{x_n}{\|x_n\|}.$ For $\frac{x_n}{\|x_n\|} \in S_{X_n}$, gap convergence (property 2 from Subsection~\ref{subsec:gaptopology}) yields $v_n \in S_X$ with
\[\bigg\|\frac{x_n}{\|x_n\|}-v_n\bigg\|_{Y} \le \sup_{z \in S_{X_n}}dist(z,S_{X}) \to 0.\]
Set $u_n=\|x_n\|v_n \in X$. Since $\|u_n\|_{Y}=\|x_n\|_{Y} \le M$, and 
\[\|x_n-u_n\|_{Y}=\|x_n\|_{Y}\bigg\|\frac{x_n}{\|x_n\|}-v_n\bigg\|_{Y} \le M \cdot \sup_{z \in S_{X_n}}dist(z,S_{X}) \to 0.\]
Hence $(x_n)_\cU = (u_n)_\cU \in X_\cU$. The reverse inclusion $(\supseteq)$ follows symmetrically using property 1.
\end{remark}
 
\section{Proof of Main Results}
We begin by establishing the full limit theorem for the numerical index under the operator opening topology, which serves as the formal proof for Theorem~\ref{thm:operatorgapconvergence}. The core advantage of this topology is algebraic in that there are near-identity isomorphisms, which allow us to conjugate optimal operators on $X_n$ to valid test operators on $X$ (and vice versa). By strictly controlling the condition numbers of these conjugating maps via Lemma~\ref{lem:bounds}, we ensure that the numerical radius and the operator norm are asymptotically preserved, circumventing the need for global Hahn-Banach extensions altogether.  
\begin{thm}\label{thm:operator_opening_convergence}
    Let $Y$ be a Banach space and let $\{X_n\}_{n \in \mathbb{N}}$ and $X$ be closed subspaces of $Y$. If $X_n \to X$ in the operator opening topology, then 
    \[
        \lim_{n \to \infty} n(X_n) = n(X).
    \]
\end{thm}
\begin{proof}
Let $1>\varepsilon > 0$. Since $X_n \to X$ in the operator opening topology, $\exists ~n_0 \in \mathbb{N}$ s.t $\forall n \ge n_0$, $r(X_n, X) < \frac{\eta}{2}$, where $\eta=\min\{\frac{\varepsilon}{2}, \frac{\varepsilon^2}{2}\}$. This in turn implies that $r_0(X_n, X) < \frac{\eta}{2}$ and $r_0(X, X_n) < \frac{\eta}{2}$. Since $\varepsilon < 1$, there exists $ C_n, D_n \in GL(Y)$ such that 
\[C_n(X_n)=X, \quad  \|C_n-I\|<\frac{\eta}{2}, \quad D_n(X)=X_n, \quad \|D_n-I\|<\frac{\eta}{2}.\]
\textbf{Claim~1:} $\limsup_{n \to \infty} n(X_n) \le n(X)$. 
 
\noindent Choose an operator $T \in \mathcal{L}(X)$ with $\|T\| = 1$ such that $v(T) < n(X) + \varepsilon$. For each $n$, define the operator $T_n \in \mathcal{L}(X_n)$ by 
    \[
        T_n := C_n^{-1} \circ T \circ C_n\big|_{X_n}.
    \]
Since $T=C_n\circ T_n\circ C_n^{-1}$, we see that $1=\|T\|\le \|C_n\|\|T_n\|\|C_n^{-1}\|$ which together with Lemma~\ref{lem:bounds} implies that
\begin{equation}
\label{eqn:boundforinverse}
\|T_n\|\ge \frac{1}{\|C_n\|\|C_n^{-1}\|}\ge\frac{2-\eta}{2+\eta}.\end{equation}
Let $(x_n, x_n^*) \in \Pi(X_n)$ be arbitrary. Define
\[x := C_n x_n \in X \quad \text{and}  \quad x^* := x_n^* \circ C_n^{-1} \in X^*.\] Observe that $x^*(x) = x_n^*(C_n^{-1} C_n x_n) = x_n^*(x_n) = 1$. Also, \begin{equation}
\label{eqn:originalevaluation}
x_n^*(T_n x_n) = x_n^*(C_n^{-1} T C_n x_n) = x^*(Tx).\end{equation}
Observe that
\[
    \|x\| \le \|C_n\| \|x_n\| \le \frac{2+\eta}{2} \quad \text{and} \quad \|x^*\| \le \|x_n^*\| \|C_n^{-1}\| \le \frac{2}{2-\eta}.
\]
Let $u := \frac{x}{\|x\|} \in S_X$ and $u^* := \frac{x^*}{\|x^*\|} \in S_{X^*}$. Then
\[
    u^*(u) = \frac{x^*(x)}{\|x\|\|x^*\|} = \frac{1}{\|x\|\|x^*\|} \ge \frac{1}{\big(\frac{2+\eta}{2}\big)\big(\frac{2}{2-\eta}\big)} = \frac{2-\eta}{2+\eta}>1-\frac{\varepsilon^2}{2}.
\]
Thus, $u^*(u) = 1 - \delta$ where $\delta < \frac{\varepsilon^2}{2}$. By Lemma~\ref{lem:bpb_sequence}, there exists a state $(y, y^*) \in \Pi(X)$ such that $\|u - y\| < \varepsilon$ and $\|u^* - y^*\| < \varepsilon$. 
Consequently,
\begin{align*}
    |u^*(Tu)| &\le |y^*(Ty)| + \|u^*\|\|T\|\|u - y\| + \|u^* - y^*\|\|T\|\|y\| \\
              &<  v(T) + 2\varepsilon.
\end{align*}
Relating this back to equation~\eqref{eqn:originalevaluation}, we see that
\[
    |x_n^*(T_n x_n)| = |x^*(Tx)| = \|x\|\|x^*\| |u^*(Tu)| \le \frac{2+\eta}{2-\eta} \big( v(T) + 2\varepsilon \big).
\]
Since $(x_n, x_n^*) \in \Pi(X_n)$ was arbitrary, taking the supremum on both sides gives \[v(T_n) \le \frac{2+\eta}{2-\eta} \big( v(T) + 2\varepsilon \big).\]
Since $n(X_n) \le \frac{v(T_n)}{\|T_n\|}$, using \eqref{eqn:boundforinverse} yields
\begin{align*}
n(X_n) &\le \frac{v(T_n)}{\|T_n\|} \\&\le  \frac{2+\eta}{2-\eta}\cdot\frac{1}{\|T_n\|} \big( v(T) + 2\varepsilon \big)\\&\le  \left(\frac{2+\eta}{2-\eta}\right)^2 \big( v(T) + 2\varepsilon \big)\\&\le \left(\frac{2+\eta}{2-\eta}\right)^2(n(X)+3\varepsilon).
\end{align*}
Taking the limit superior as $n \to \infty$ on both sides, we get that
\[
    \limsup_{n \to \infty} n(X_n) \le \left(\frac{2+\eta}{2-\eta}\right)^2( n(X) + 3\varepsilon).
\]
Letting $\varepsilon \to 0$ gives us that $\limsup_{n \to \infty} n(X_n) \le n(X)$.
 
\medskip
\noindent\textbf{Claim~2:} $\liminf_{n \to \infty} n(X_n) \ge n(X)$.
 
\noindent For each $n \ge n_0$, choose an operator $T_n \in \mathcal{L}(X_n)$ with $\|T_n\| = 1$ such that $v(T_n) < n(X_n) + \varepsilon$. 
Define $T \in \mathcal{L}(X)$ by
\[
    T := D_n^{-1} \circ T_n \circ D_n\big|_X.
\]
Using Lemma~\ref{lem:bounds}, as before, 
\[\|T\| \ge \frac{\|T_n\|}{\|D_n\| \|D_n^{-1}\|} \ge \frac{2-\varepsilon}{2+\varepsilon}.\]
Let $(x, x^*) \in \Pi(X)$ be an arbitrary state. Define $x_n := D_n x \in X_n$ and $x_n^* := x^* \circ D_n^{-1} \in X_n^*$. 
We have $x_n^*(x_n) = 1$ and $x^*(Tx) = x_n^*(T_n x_n)$. 
Observe that \[\|x_n\| \le \|D_n\|\le \frac{2+\eta}{2} \quad \text{and}\quad  \|x_n^*\| \le \frac{2}{2-\eta}.\] Letting $u_n := \frac{x_n}{\|x_n\|}$ and $u_n^* = \frac{x_n^*}{\|x_n^*\|}$, we see that  
$$u_n^*(u_n) \ge \frac{2-\eta}{2+\eta} = 1 - \frac{2\eta}{2+\eta}>1-\frac{\varepsilon^2}{2}.$$ 
Using Lemma~\ref{lem:bpb_sequence}, there exists $(y_n, y_n^*) \in \Pi(X_n)$ such that $\|u_n - y_n\| < \varepsilon$ and $\|u_n^* - y_n^*\| < \varepsilon$. Therefore,
\[
    |u_n^*(T_n u_n)| \le |y_n^*(T_n y_n)| + 2\varepsilon\le v(T_n) + 2\varepsilon.
\]
Therefore, we can bound our initial evaluation:
\[
    |x^*(Tx)| = \|x_n\|\|x_n^*\| |u_n^*(T_n u_n)| \le \frac{2+\eta}{2-\eta} \big( v(T_n) + 2\varepsilon\big).
\]
Taking the supremum over all $(x, x^*) \in \Pi(X)$ gives $v(T) \le \frac{2+\varepsilon}{2-\varepsilon} \big( v(T_n) + 2\varepsilon \big)$. 
Since $n(X) \le \frac{v(T)}{\|T\|}$, we get that
\[
    n(X) \bigg(\frac{2-\varepsilon}{2+\varepsilon}\bigg) \le v(T) \le \frac{2+\varepsilon}{2-\varepsilon} \big( n(X_n) + 3\varepsilon \big).
\]
Consequently,
\[
    n(X) \bigg(\frac{2-\varepsilon}{2+\varepsilon}\bigg)^2 - 3\varepsilon \le n(X_n).
\]
Taking the limit inferior as $n \to \infty$ gives us that
\[
\liminf_{n \to \infty} n(X_n)\ge n(X) \bigg(\frac{2-\varepsilon}{2+\varepsilon}\bigg)^2 - 3\varepsilon.
\]
Since $\varepsilon \in (0,1)$ was arbitrary, letting $\varepsilon \to 0$ gives $n(X) \le \liminf_{n \to \infty} n(X_n)$. The claim follows. Combining Claim~1 and Claim~2 yields 
\[\lim_{n \to \infty}n(X_n)=n(X).\]
\end{proof}  
It is important to emphasize that the families of subspaces considered in \cite[Theorem~5.1]{martin2011numerical} do not, in general, converge in the operator opening topology, as the following example demonstrates. 
\begin{example}
    Let $Z=\ell_2$ and for each $m \in \mathbb{N}$ let 
    \[Z_m:=\text{span}\;\{e_1,e_2, \hdots, e_m\},\]
    where $(e_k)_{k \ge 1}$ is the standard orthonormal basis of $\ell_2$. Then $(Z_m)_{m \in \mathbb{N}}$ satisfies that $\overline{\bigcup_{m\in \mathbb{N}}Z_m}=\ell_2.$
    Moreover, there exists a projection $P_m:\ell_2 \to Z_m$, and therefore, each $Z_m$ is $1$-complemented. However, $(Z_m)_{m \in \mathbb{N}}$ does not converge to $\ell_2$ in operator opening topology. Suppose, by way of contradiction, $Z_m \to Z$ in operator opening topology. Then there exists an invertible operator $C \in GL(\ell_2)$ such that $C(Z_m)=\ell_2$. Thus
    \[Z_m=C^{-1}(C(Z_m))=C^{-1}(\ell_2)=\ell_2,\] which is a contradiction to the fact that $Z_m$ is finite dimensional and proper subspace of $\ell_2$. Therefore the set $\{C \in GL(\ell_2): C(Z_m)=\ell_2\}$ is empty, and hence by definition $r_0(Z_m, \ell_2)=1$. Therefore, $Z_m \nrightarrow \ell_2$ in operator opening topology.
\end{example}
It is worth noting that the topologies induced on $G(X)$ by the geometric and operator openings are generally different. As demonstrated in \cite{gurariigeometric}, if $K$ is a complemented subspace of $Y$ and $L$ is an isomorphic uncomplemented subspace of $Z$, the geometric and operator opening topologies induced on $G(Y \oplus Z)$ do not coincide.
 
While the operator opening topology yields clean convergence results, it is too restrictive for many natural approximation schemes. The weaker gap topology is more widely applicable, but the absence of near-identity isomorphisms prevents direct operator transport between subspaces. To bridge this gap, we now develop an ultraproduct approach. The key insight is that gap topology convergence $X_n \to X$ induces a canonical identification at the ultraproduct level (Remark~\ref{rem:gap_identification}), allowing us to compare numerical indices without explicitly transporting operators. The technical engine of this approach is the following theorem, which establishes that the numerical radius behaves well under ultraproducts.
 
\begin{thm}\label{thm:numerical_radius_ultraproduct}
Let $\{X_n\}_{n \in \N}$ be Banach spaces, $\cU$ a free ultrafilter on $\N$, and $\{T_n\}$ a uniformly bounded family with $T_n \in \mathcal{L}(X_n)$. Then \[v((T_n)_\cU) = \lim_\cU v(T_n).\]
\end{thm}
\begin{proof}
Set $T = (T_n)_\cU$ and $Z = \prod_\cU X_n$ and $M=\sup_{n}\|T_n\|< \infty$. We use throughout that if $\lim_{\cU}a_n=a$ and $a>c$, then $a_n>c$ for $\cU$-almost all $n$, and that representatives of elements of $\prod_{\cU}X_n$ may be modified on a $\cU$-null set without changing the equivalence class. 

\noindent We first establish $v(T) \ge \lim_\cU v(T_n)$. Let $\varepsilon > 0$. For each $n$, choose $(x_n, x_n^*) \in \Pi(X_n)$ with $|x_n^*(T_n x_n)| > v(T_n) - \varepsilon$. Set $\mathbf{x} = (x_n)_\cU \in S_Z$ and $\mathbf{f}=\Phi((x_n^*)_\cU) \in {Z^*}$. Since $\|x_n^*\|=1$ for all $n$, we have $\|\mathbf{f}\|\le 1$. Also $\mathbf{f(x)}=\lim_{\cU}x_n^*(x_n)=1$ forces $\|\mathbf{f}\|=1$, so $(\mathbf{x},\mathbf{f}) \in \Pi(Z)$.  Since $|\cdot|$ is continuous and ultralimits commute with the continuous functions on bounded sets, 
  \[  v(T) \ge |\mathbf{f}(T\mathbf{x})| = |\lim_\cU x_n^*(T_n x_n)| = \lim_\cU |x_n^*(T_n x_n)| \ge \lim_\cU v(T_n) - \varepsilon,
\]
Letting $\varepsilon \to 0$ gives $v(T) \ge \lim_\cU v(T_n)$.
 
We now prove $v(T) \le \lim_\cU v(T_n)$. Let $\varepsilon \in (0,1)$ and $(\mathbf{x}, \mathbf{g}) \in \Pi(Z)$ with $\mathbf{x}=(x_n)_{\cU}$. Since $\lim_{\cU}\|x_n\|=1$, replacing $x_n$ by $\frac{x_n}{\|x_n\|}$ on the set $\{n: x_n \ne 0\} \in \cU$ does not change the equivalence class, so we may assume that $\|x_n\| = 1$ for all $n$. By equation~\eqref{eq:norming}, $\Phi(\prod_\cU X_n^*)$ is $1$-norming for $Z$, and the bipolar theorem gives $\overline{\Phi(B_{\prod_\cU X_n^*})}^{w^*} = B_{Z^*}$ (see~\cite[Section~1]{heinrich1980}). So there exists $(g_n)_{\cU} \in \prod_{\cU}X_n^*$ with $\lim_{\cU}\|g_n\|\le 1$ such that  
\begin{equation}\label{eq:approx_value}
    |\lim_\cU g_n(x_n) - 1| < \frac{\varepsilon^2}{4} \quad \text{and} \quad |\Phi((g_n)_\cU)(T\mathbf{x}) - \mathbf{g}(T\mathbf{x})| < \varepsilon.
\end{equation}
 
From the first inequality in ~\eqref{eq:approx_value}, $\lim_\cU \text{Re}\, g_n(x_n) > 1 - \frac{\varepsilon^2}{4}$, and since $\text{Re}\, g_n(x_n) \le \|g_n\|$ pointwise, we get $\lim_{\cU}\|g_n\|>0$. Hence $\|g_n\|>0$ for $\cU$-almost all $n$, and set $\hat{g}_n = g_n/\|g_n\| \in S_{X_n^*}$. Since $\lim_{\cU} \|g_n\|>0$, we have
\[
    \lim_\cU \text{Re}\, \hat{g}_n(x_n) = \frac{\lim_\cU \text{Re}\, g_n(x_n)}{\lim_\cU \|g_n\|} \ge \frac{1 - \frac{\varepsilon^2}{4}}{1} > 1 - \frac{\varepsilon^2}{2}.
\]
Therefore $\text{Re}\, \hat{g}_n(x_n) > 1 - \frac{\varepsilon^2}{2}$ for $\cU$-almost all $n$. Lemma~\ref{lem:bpb_sequence} then yields $(y_n, y_n^*) \in \Pi(X_n)$ with 
\begin{equation}\label{BPBforyn}
\|x_n - y_n\| < \varepsilon \quad \text{and} \quad \|\hat{g}_n - y_n^*\| < \varepsilon \quad \text{for} \; \cU\;\text{-almost all}\; n.
\end{equation}
Writing $\hat{g}_n(T_n x_n)= (\hat{g}_n-y_n^*)(T_nx_n)+y_n^*(T_n(x_n-y_n))+y_n^*(T_ny_n)$ and using $\|x_n\|=\|y_n^*\|=1, \|T_n\| \le M, |y_n^*(T_ny_n)| \le v(T_n)$, and ~\eqref{BPBforyn},
\[
    |\hat{g}_n(T_n x_n)| \le M\|\hat{g_n}-y_n^*\|+M\|x_n-y_n\|+v(T_n)<2M\varepsilon+v(T_n).
\]
Since $\|g_n\| \le 1 + \varepsilon$ for $\cU$-almost all $n$, we obtain 
$$|g_n(T_n x_n)| = \|g_n\||\hat{g_n}(T_nx_n)|\le (1 + \varepsilon)(v(T_n) + 2M\varepsilon) \quad \text{for}\, \cU\text{-almost all}\, n.$$ 
Taking ultralimits and using the continuity of $|\cdot|$, 
\[\lim_{\cU}|g_n(T_nx_n)| \le (1+\varepsilon)(\lim_{\cU}v(T_n)+2M\varepsilon).\]
Since $\Phi((g_n)_{\cU})(T\mathbf{x})=\lim_{\cU}g_n(T_nx_n)$, the triangle inequality and ~\eqref{eq:approx_value} give
\[|\mathbf{g}(T\mathbf{x})| \le \varepsilon + \lim_{\cU}|g_n(T_nx_n)|\le \varepsilon + (1+\varepsilon)\left(\lim_{\cU}v(T_n)+2M\varepsilon\right).\]
Since $(\mathbf{x}, \mathbf{g}) \in \Pi(Z)$ was arbitrary, $v(T) \le \varepsilon + (1+\varepsilon)\left(\lim_{\cU}v(T_n)+2M\varepsilon\right)$, and letting $\varepsilon \to 0$, we conclude $v(T) \le \lim_\cU v(T_n)$.
\end{proof}
 
The exact preservation of the numerical radius under ultraproducts has immediate consequences for the numerical index. Since the numerical index is defined as an infimum of numerical radii over operators of norm one, and the ultraproduct operator norm satisfies $\|(T_n)_\cU\| = \lim_\cU \|T_n\|$, we obtain a one-sided bound.
 
\begin{cor}\label{cor:index_upper}
 For any Banach spaces $\{X_n\}$ and free ultrafilter $\cU$,
 \[n\bigg(\prod_\cU X_n\bigg) \le \lim_\cU n(X_n).\]
\end{cor}
\begin{proof}
For each $n$, choose $T_n \in \mathcal{L}(X_n)$ with $\|T_n\| = 1$ and $v(T_n) < n(X_n) + 1/n$. Since $\|(T_n)_{\cU}\|=\lim_{\cU}\|T_n\|=1$, Theorem~\ref{thm:numerical_radius_ultraproduct} gives
\[n\left(\prod_\cU X_n\right) \le v((T_n)_{\cU})= \lim_{\cU}v(T_n) \le \lim_{\cU} \left(n(X_n)+\frac{1}{n}\right)=\lim_{\cU}n(X_n). \]
\end{proof}
 
Specializing to the case where all spaces coincide, we obtain a bound for ultrapowers.
 
\begin{cor}\label{cor:ultrapower}
For any Banach space $X$ and free ultrafilter $\cU$,
\[n(X_\cU) \le n(X).\]
\end{cor}
\begin{proof}
Apply Corollary~\ref{cor:index_upper} with $X_n = X$ for all $n$.
\end{proof}
We now connect these ultraproduct bounds to the gap topology convergence problem. The key observation is that gap convergence $X_n \to X$ inside an ambient space $Y$ induces the identification $\prod_\cU X_n \cong X_\cU$ inside $Y_\cU$ (Remark~\ref{rem:gap_identification}). This allows us to reduce the gap topology convergence question to a purely intrinsic property of ultrapowers.
 
\begin{thm}\label{thm:reduction}
If $X_n \to X$ in gap topology inside $Y$. Then
\begin{enumerate}[label=(\roman*)]
    \item $n(X_\cU) \le \lim_\cU n(X_n)$ for every free ultrafilter $\cU$.
    \item If $n(X_\cU) = n(X)$, for every free ultrafilter $\cU$, then $n(X) \le \displaystyle\liminf_{n \to \infty} n(X_n)$.
\end{enumerate}

\end{thm}
\begin{proof}
Part~(i) follows from Remark~\ref{rem:gap_identification} and Corollary~\ref{cor:index_upper}. For part~(ii), let $\{X_{k_j}\}$ be a subsequence with $n(X_{k_j}) \to \liminf_{n}n(X_n)$. The family $\{k_j: j \ge N\}_{N \in \N}$ has finite intersection property, hence by the ultrafilter lemma, it extends to a free ultrafilter $\cU$ on $\N$ containing each set $\{k_j: j \ge N\}$. Since $X_n \to X$ in gap topology, we have from part~(i),
\[n(X)=n(X_{\cU}) \le \lim_\cU n(X_n)=\lim_{j}n(X_{k_j})=\liminf_{n \to \infty}n(X_n).\]
\end{proof}
 
Theorem~\ref{thm:reduction} reduces the gap topology convergence problem to the following fundamental question about the stability of the numerical index under ultrapowers.
 
\begin{question}\label{q:ultrapower}
Does $n(X_\cU) = n(X)$ hold for every Banach space $X$ and every free ultrafilter $\cU$?
\end{question}
 
An affirmative answer to Question~\ref{q:ultrapower} would immediately imply, via Theorem~\ref{thm:reduction}, that the numerical index is \emph{lower semicontinuous} under the gap topology convergence, that is, $n(X) \le \liminf_{n} n(X_n)$ whenever $X_n \to X$ in gap topology. Full continuity would require an upper bound $\limsup_{n}n(X_n) \le n(X)$, which is a separate question. We conclude by recording the cases where Question~\ref{q:ultrapower} is known to have an affirmative answer and discussing why the general case remains delicate.
 
\begin{remark}\label{rem:ultrapower_cases}
Question~\ref{q:ultrapower} has an affirmative answer in the following cases:
\begin{enumerate}
    \item When $X$ is lush, since lushness is preserved by ultraproducts~\cite[Corollary~4.5]{BKMM2009}, so $n(X_\cU) = n(X) = 1$.
    \item When $n(X) = 0$: Corollary~\ref{cor:ultrapower} gives $0 \le n(X_{\cU})\le n(X)=0$, and hence $n(X_{\cU})=0$.
\end{enumerate}
The general case remains open. We also note that the numerical index is not a local property: the inequality $n(X) > n(X^*)$ can occur~\cite{BKMW2007}. This suggests that Question~\ref{q:ultrapower} may have a subtle answer in general, as the ultrapower construction can interact nontrivially with duality and local structure.
\end{remark}

\bibliography{name}

@article{BKMM2009,
  author  = {Boyko, K. and Kadets, V. and Mart{\'\i}n, M. and Mer{\'\i}, J.},
  title   = {Properties of lush spaces and applications to {B}anach spaces with numerical index~$1$},
  journal = {Studia Math.},
  volume  = {190},
  year    = {2009},
  pages   = {117--133},
}

@article{BKMW2007,
  author  = {Boyko, K. and Kadets, V. and Mart{\'\i}n, M. and Werner, D.},
  title   = {Numerical index of {B}anach spaces and duality},
  journal = {Math. Proc. Cambridge Philos. Soc.},
  volume  = {142},
  year    = {2007},
  pages   = {93--102},
}

@article{Bollob1970,
  author  = {Bollob{\'a}s, B.},
  title   = {An extension to the theorem of {B}ishop and {P}helps},
  journal = {Bull. London Math. Soc.},
  volume  = {2},
  year    = {1970},
  pages   = {181--182},
}

@book{BonsallDuncan1971,
  author    = {Bonsall, F.F. and Duncan, J.},
  title     = {Numerical Ranges of Operators on Normed Spaces and of Elements of Normed Algebras},
  series    = {London Math. Soc. Lecture Note Ser.},
  volume    = {2},
  publisher = {Cambridge Univ. Press},
  year      = {1971},
}

@book{BonsallDuncan1973,
  author    = {Bonsall, F.F. and Duncan, J.},
  title     = {Numerical Ranges {II}},
  series    = {London Math. Soc. Lecture Note Ser.},
  volume    = {10},
  publisher = {Cambridge Univ. Press},
  year      = {1973},
}

@article{heinrich1980,
  author  = {Heinrich, S.},
  title   = {Ultraproducts in {B}anach space theory},
  journal = {J. Reine Angew. Math.},
  volume  = {313},
  year    = {1980},
  pages   = {72--104},
}

@article{martin2011numerical,
  author  = {Mart{\'\i}n, M. and Mer{\'\i}, J. and Pay{\'a}, R.},
  title   = {Numerical index of absolute sums of {B}anach spaces},
  journal = {J. Math. Anal. Appl.},
  volume  = {375},
  year    = {2011},
  pages   = {207--222},
}

@article{aksoy2013limit,
  author  = {Aksoy, A.G. and Lewicki, G.},
  title   = {Limit theorems for the numerical index},
  journal = {J. Math. Anal. Appl.},
  volume  = {398},
  year    = {2013},
  pages   = {296--302},
}

@article{ostrovskii1994topologies,
  author  = {Ostrovskii, M.I.},
  title   = {Topologies on the set of all subspaces of a {B}anach space and related questions of {B}anach space geometry},
  journal = {Quaestiones Math.},
  volume  = {17},
  year    = {1994},
  pages   = {259--319},
}

@article{HIspaces,
  author  = {Gowers, W.T. and Maurey, B.},
  title   = {The unconditional basic sequence problem},
  journal = {J. Amer. Math. Soc.},
  volume  = {6},
  year    = {1993},
  pages   = {851--874},
}

@book{LindenstraussTzafriri1977,
  author    = {Lindenstrauss, J. and Tzafriri, L.},
  title     = {Classical {B}anach Spaces {I}: Sequence Spaces},
  publisher = {Springer-Verlag},
  address   = {Berlin-Heidelberg-New York},
  year      = {1977},
}

@article{dantas2022bishopold,
  title={The Bishop--Phelps--Bollob{\'a}s theorem: an overview},
  author={Dantas, Sheldon and Garc{\'\i}a, Domingo and Maestre, Manuel and Rold{\'a}n, {\'O}scar},
  journal={Operator and Norm Inequalities and Related Topics},
  pages={519--576},
  year={2022},
  publisher={Springer}
}

@article{gurariigeometric,
  author  = {Gurarii, V.I.},
  title   = {On openings and inclinations of subspaces of {B}anach spaces},
  journal = {Teor. Funkci{\u\i} Funkcional. Anal. i Prilo\v{z}en.},
  volume  = {1},
  year    = {1965},
  pages   = {194--204},
}
\bibliographystyle{amsalpha}
\end{document}